\documentclass[12pt]{amsart}

\usepackage{amssymb,latexsym}
\usepackage{amsmath,amsfonts,amsthm,amscd,amsxtra,enumerate}
\usepackage{mathtools}
\usepackage[all]{xy}
\usepackage{amsmath,latexsym,amssymb, enumerate, amsthm, epsfig, hyperref,tikz} 
\usepackage[margin=1in]{geometry}
\usepackage{epsfig}

\usepackage[utf8]{inputenc}

\newtheorem{conjeture}{ Conjecture}[section]
\newtheorem{definition}[conjeture]{ Definition}
\newtheorem{notation}[conjeture]{ Notation}
\newtheorem{example}[conjeture]{ Example}
\newtheorem{theorem}[conjeture]{ Theorem}
\newtheorem{lemma}[conjeture]{ Lemma}
\newtheorem{remark}[conjeture]{ Remark}
\newtheorem{corollary}[conjeture]{ Corollary}
\newtheorem{proposition}[conjeture]{ Proposition}

\renewcommand\dim{\text{\rm dim}}

\begin{document}
\setlength{\abovedisplayskip}{3pt}
\setlength{\belowdisplayskip}{3pt}

\title[Certain Homological invariants of bipartite kneser graphs] {Certain Homological invariants of bipartite kneser graphs}

\author[A. Kumar]{Ajay Kumar}
\address{Department of Mathematics, Indian Institute of Technology Jammu, Jammu, J$\&$K, India}
\email{ajay.kumar@iitjammu.ac.in}

\author[P. Singh]{Pavinder Singh}
\address{Department of Mathematics, Central University of Jammu, Rahya-Suchani(Bagla), Samba-181143, J$\&$K, India}
\email{pavinders@gmail.com}

\author[R. Verma]{Rohit Verma}
\address{Department of Mathematics, Central University of Jammu, Rahya-Suchani(Bagla), Samba-181143, J$\&$K, India}
\email{rhtgm@yahoo.in}

\subjclass[2010]{13D02, 13F55, 05C69, 05C99}
\keywords{Edge ideals, bipartite Kneser graphs, graded Betti numbers, regularity, projective dimension}
\date{}


\begin{abstract}
In this paper, we obtain a combinatorial formula for computing the Betti numbers in the linear strand of edge ideals of bipartite Kneser graphs. We deduce lower and upper bounds for regularity of powers of edge ideals of these graphs in terms of associated combinatorial data and show that the lower bound is attained in some cases. Also, we obtain bounds on the projective dimension of edge ideals of these graphs in terms of combinatorial data.
\end{abstract}
\maketitle

\maketitle
\section{Introduction}
Let $R=\Bbbk[x_1,x_2,\ldots,x_n]$ be a standard graded polynomial ring over an infinite  field $\Bbbk$ and $I$ be a homogeneous ideal in $R$.   Then the minimal graded free resolution of the quotient ring $R/I$ is given by
\[
0\rightarrow\bigoplus_jR[-j]^{\beta_{\ell,j}}\rightarrow\cdots\rightarrow\bigoplus_jR[-j]^{\beta_{i,j}}\rightarrow\cdots\rightarrow\bigoplus_jR[-j]^{\beta_{1,j}}\rightarrow R  \rightarrow R/I\rightarrow 0
\] 
where $R[-j]$ is a graded free $R$-module of rank one whose $d$th graded component is $R[-j]_d=R_{d-j}$ the $(d-j)$th graded component of $R$. The non-negative integer $\beta_{i,j}$ is known as the $i$th \textit{graded Betti number} of $R/I$ in degree $j$ which is same as the number of generators of degree $j$ in the $i$th syzygy module of $R/I$.
The \emph{regularity} of  $I$, denoted as $\text{reg}(I)$, is defined as $$\text{reg}(I)=\max\{j-i:\beta_{i,j}(I)\neq 0\}.$$
The \emph{projective dimension} of $I$, denoted as $\text{pd}(I)$, is defined as $$\text{pd}(I)=\max\{i:\beta_{i,j}(I)\neq 0 ~\text {for ~ some}~ j\}.$$
There is a natural correspondence between finite simple graphs on $n$ vertices and quadratic square-free monomial ideals in $R$. For every finite simple graph $G=(V,E)$ (a graph without loops and multiple edges), with vertex set $V$ in bijection with the set of variables $x_1,x_2,\ldots,x_n$ in the polynomial ring $R$, the {\it edge ideal} $I(G)$ of $G$ is the quadratic square-free monomial ideal in $R$ generated by all monomials $x_ix_j$ such that $\{x_i,x_j\}$ is an edge of $G$. For a graph $G,$ we write $\textrm{pd}(G), \beta_{i,j}(G)$ and $\textrm{reg}(G)$ as shorthand for $\textrm{pd}(R/I(G)), \beta_{i,j}(R/I(G))$ and $\textrm{reg}(R/I(G)),$ respectively. The edge ideals were first introduced and studied by Villarreal~\cite{Vill}. The problem of computing and finding bounds on the homological invariants such as Betti numbers, regularity, projective dimension of edge ideals and their powers in terms of the combinatorial invariants of the associated graphs have been studied by a number of researchers (see ~\cite{AB,CH,huneke,dao,hatu,hatu1,hawo,jacq,janase,jase,VK,mogh,pash,paro,wood}). It was shown by Cutkosky, Herzog and Trung \cite{CH}, and independently by Kodiyalam \cite{VK}, that for a homogeneous ideal $I$, $\text{reg}(I^p)$  is asymptotically a linear function for $p>>0$, i.e., there exist integers $a,b$ and $p_0$ such that $$\text{reg}(I^p)=ap+b$$ for all $p \geq p_0$. It is known that the constant $a$ is bounded by the maximum degree of minimal generators of $I$. However, constants $b$ and $p_0$ are not well understood. 
When $I = I(G)$ is an edge ideal, then reg$(I(G)^p)=2p+b$ for all $p\geq p_0$. In this article, we explicitly compute $b$ and $p_0$ for some particular classes of bipartite Kneser graphs.

It is known that for a graph $G$,
$$ \text{ind}(G)+1 \leq \text{reg}\left(I(G)\right) \leq \text{cochord}(G)+1,$$ where $\text{ind}(G)$ is the induced matching number of $G$ and $\text{cochord}(G)$ is the co-chordal number of $G$ (see \cite{katz,wood}). Beyarslan, H\`{a} and Trung~\cite{behatr} shown that for a finite simple graph $G$ on $n$ vertices containing a Hamiltonian path, the regularity of edge ideal $I(G)$ satisfies the following inequality 
 \begin{equation}\label{EQ}
 \text{reg}(I(G))\le 
 \begin{array}{ll}
 \left\lfloor\frac{n+1}{3}\right\rfloor+1. 
 \end{array}
 \end{equation}
 
 In this article, we study various homological invariants such as Betti numbers, regularity, projective dimension of edge ideals of bipartite Kneser graphs. Given two positive integers $m$ and $k$ with $m \geq 2k,$ the bipartite Kneser graph $\mathcal{H}(m,k)$ is the graph whose vertices are  $k$-subsets and $(m-k)$-subsets of $[m]=\{1,2,\ldots,m\}$ such that any two vertices are connected by an edge if and only if one is a subset of the other. Determining whether a graph G has a hamiltonian cycle is an NP-Complete problem  and it has long been conjectured that for any $k \geq 1$ and $m \geq 2k + 1,$ all
 bipartite Kneser graphs  have a Hamiltonian cycle. The motivation for this conjecture is a more general conjecture
 due to Lovász \cite{LL}, which states that, apart from five exceptional graphs, every connected vertex-transitive graph has a Hamiltonian cycle. A
 graph is called a vertex-transitive graph if it ‘looks the same’ from the point of view of any vertex. M{\"u}tze and Su in ~\cite{musu} settled the conjecture of Lovász for bipartite Kneser graphs, that is, for any $k \geq 1$ and $m \geq 2k + 1,$ all
 bipartite Kneser graphs  have a Hamiltonian cycle. Since every Hamiltonian cycle contains a  Hamiltonian path, equation~(\ref{EQ}) gives an upper bound on the regularity of edge ideals of bipartite Kneser graphs.
 
 Beyarslan, H\`{a} and Trung~\cite{behatr} proved that for every finite simple graph $G$ and every integer $p\geq 1$, the following inequality 
 \begin{equation}\label{eqn1}
 \text{reg}(R/{I(G)}^p)\geq 2(p-1)+\text{ind}(G)
 \end{equation}
 holds. Also, they have shown that the equality holds, if $G$ is a forest or a cycle. Jayanthan, Narayanan and Selvaraja in~\cite[Theorem 3.6]{janase} gave an upper bound for powers of edge ideals of bipartite graphs as follows
  \begin{equation}\label{eqn}
 {\rm reg}(R/{I(G)}^p)\leq 2(p-1)+{\rm cochord}(G)
\end{equation} for all $p \geq 1$. 

In \cite{behatr}, the authors raised the question, for which graphs $\textrm{reg}(I(G))^p=2p+\text{ind}(G)-1$ for $p>>0$. It is known that for certain classes of  graphs $\textrm{reg}(I(G))^p=2p+\text{ind}(G)-1$.

\begin{theorem}
	Let $G$ be a simple graph and $\text{ind}(G)$ be the induced matching number of $G.$
	Then for all $p \geq 1,$ we have
	$reg (I(G)^p) = 2p + \text{ind}(G) - 1$ in the following cases:
	\begin{enumerate}[(a)]
		\item $G$ is a unmixed bipartite graph (see \cite{janase});
		\item $G$ is a weakly chordal bipartite graph (see \cite{janase});
		\item $G$ is a $P_6$ free bipartite graph (see \cite{janase});
		\item $G$ is a whiskered bipartite graph (see \cite{janase});
		\item $G$ is a very well-covered graph (see \cite{js}).
		\item $G$ is a forest (see \cite{behatr}).
	\end{enumerate}
Also, authors in \cite{behatr} prove that for all $p \geq 2,$ and $G$ a $n$-cycle, 
$reg( I(G)^p )= 2p + \text{ind}(G) - 1$.
\end{theorem}

It is easy to see that for all $k \geq 1$, $G=\mathcal{H}(2k+1,k)$ does not belongs to any of classes mentined above. We prove that for $p \geq 1,$ and $G=\mathcal{H}(2k+1,k)$ we have
$reg (I(G)^p) = 2p + \text{ind}(G) - 1.$
Using equations~\ref{eqn1} and ~\ref{eqn},  we deduce lower and upper bounds for regularity of powers of edge ideals of bipartite Kneser graphs in terms of associated combinatorial data, and show that the lower bound is attained in some cases.
\begin{theorem}{\rm (Theorem \ref{regb})}
	For a bipartite Kneser graph $\mathcal{H}(m,k)$, we have
	\[
	2(p-1)+{2k\choose k}\leq {\rm reg}(R/{I(\mathcal{H}(m,k))}^p)\leq 2(p-1)+{m\choose k}.
	\] 
	Furthermore, the lower bound is attained if $m=2k$  or $2k+1$.                                                                        
\end{theorem}

 In Section $3$, we deduce a combinatorial formula for computing the graded Betti numbers in the linear strand of edge rings of bipartite Kneser graphs using Hochster's formula given as follows. 

\begin{theorem}{\rm (Theorem \ref{mainb})} 
	For the bipartite Kneser graph $\mathcal{H}(m,k)$, we have 
	\[
	\beta_{i,i+1}(\mathcal{H}(m,k))=\sum_{r+s=i+1}\sum_{t=k}^{m-k}\binom{{t\choose k}}{r}N_{m,s,m-k}^{=t},
	\]
	where $N_{m,s,m-k}^{=t}$ is defined as in the Theorem \ref{mainb}.
\end{theorem}

 Authors in  \cite{huneke, dao} studied the bounds on the projective dimension of graphs using domination parameters. We obtain bounds on the projective dimension of edge ideals of bipartite Kneser graphs. 
\begin{theorem}{\rm (Theorem \ref{pdmain})} 
	For the bipartite Kneser graph $\mathcal{H}(m,k)$, we have
	\[  
	2{m \choose k}-{2k \choose k} \leq {\rm pd}(\mathcal{H}(m,k))\leq 2{m \choose k}-\max \left\{ k+1,\frac{{m \choose k}}{{m-k \choose k}}\right\}.
	\]
\end{theorem}

\section{Preliminaries}

{\small \subsection{Graphs and Independence Complex}
Consider a finite simple graph $G=(V(G),E(G))$} (or, simply $G=(V,E)$). Then we have the following:
 
A subgraph $H$ of graph $G$ with vertex set $V(H)=\{x_1,x_2,\ldots,x_n\}(n\geq 2)$ and edge set $E(H)=\{\{x_{i},x_{i+1}\};1\le i\le n-1\}$ is called a {\it path} of length $n$, denoted by $P_n$. If a path $P_n$ in a graph contains each vertex of $G$, then such a path $P_n$ is called a {\it Hamiltonian path}. 

 If the initial vertex $x_1$ and terminal vertex $x_n$ of path $P_n$ are also connected by an edge in $G$, then we call the subgraph $H^\prime$ with  $V(H^\prime)=\{x_1,x_2,\ldots,x_n\}(n\geq 3)$ and edge set $E(H^\prime)=\{\{x_{1},x_{2}\},\ldots,\{x_{n-1},x_{n}\},\{x_{n},x_{1}\} \}$ is called a {\it cycle} of length $n$ or $n$-cycle, denoted by $C_n$. If a cycle $C_n$ in a graph $G$ contains each vertex of $G$, then such a cycle is called a {\it Hamiltonian cycle}.

A subset $S\subseteq V(G)$ is called an {\it independent} (or {\it stable}) set if no two vertices in $S$ are connected by an edge in $G$. A graph $G$ is {\it bipartite} if its vertex set $V$ can be partitioned into two disjoint independent subsets $V_1$ and $V_2$ such that every edge of $E$ is the form $e=\{u,v\}$ with $u\in V_1$ and $v\in V_2$.

We say a graph $G$ is said to be {\it chordal} if it does not contain any induced cycle of length greater than 3 and a graph $G$ is {\it co-chordal} if its complement graph $G^c$ is chordal, where $G^c$ is a simple graph whose vertex set is same as that of vertex set of $G$ and edge set $E(G^c)=\{\{u,v\}:u,v\in V(G)~\&~\{u,v\}\notin E(G)\}$.

Two distinct edges $e=\{x_i,x_j\}$ and $e^\prime=\{y_s,y_t\}$ in $G$ are said to be {\it 3-disjoint} if $\{x_i,x_j\}\cap \{y_s,y_t\}=\emptyset$ and the induced subgraph $G[W]$ is the disjoint union of edges $e$ and $e^\prime$ where $W=\{x_i,x_j,y_s,y_t\}$. A subset $T=\{e_1,e_2,\ldots,e_t\}\subset E(G)$ is said to be {\it pairwise 3-disjoint} subset if any two distinct edges in $T$ are 3-disjoint. 

A {\it matching} $M$ in a graph $G$ is a subset of $E(G)$ such that no two edges in $M$ have a common vertex. If a matching $M$ of $G$ is pairwise $3$-disjoint, then such a matching $M$ is said to be an {\it induced matching}. The maximum of size of induced matchings in a graph $G$ is known as the {\it induced matching number} of $G$, denoted by $\text{ind}(G)$. Clearly, the induced matching number of $G$ is the largest size of pairwise 3-disjoint subset of $E(G)$. 

Let $u\in V(G)$. If $\{v_1,v_2, \ldots,v_n\}$ is the set of all vertices in $G$ such that $u$ is adjacent to $v_i$; for $1\le i\le n$, then such a set is called an {\it open neighbourhood} of $u$, denoted by $N_G(u)$, and the set 
\[
\mathcal{N}_G[u]=\{u,v_1,v_2, \ldots, v_n\}\subseteq V(G).
\]
is known as {\it closed neighbourhood} of $u$ in $G$, denoted by $\mathcal{N}_G[u]$. In such a case, the subgraph $\mathcal{S}_u$ of $G$, with  vertex set $V(\mathcal{S}_u)=\mathcal{N}_{G}[u]$ and edge set $E(\mathcal{S}_u)=\{\{u,v_i\}\mid 1\leq i\leq n\}$, is called the {\it star} with center $u$.

A collection $\{H_1,H_2,\ldots,H_s\}$ of subgraphs of $G$ is called a {\it covering} of $G$ if
\[
\bigsqcup_{j=1}^s E(H_j)=E(G).
\] 
The smallest size of the above collection of subgraphs is called {\it cover number}. If each subgraph $H_j$ in the  collection $\{H_1,H_2,\ldots,H_s\}$ of subgraphs of $G$ is co-chordal, then the corresponding cover number is called {\it co-chordal cover number}. It is denoted by ${\rm cochord}(G)$

\begin{example}\rm
Let $G$ and $G^{\prime}$ be the two bipartite graphs given in the Figure~\ref{g} and Figure~\ref{gg}. Clearly from the Figure~\ref{g}, ${\rm cochord}(G)=3$ and from the Figure~\ref{gg}, ${\rm cochord}(G^{\prime})=2$.   
\begin{figure}[ht]
\centering
\begin{minipage}[b]{0.35\linewidth}
\centering
\begin{tikzpicture}[scale=4]
\tikzstyle{every node}=[circle, fill=black!,inner sep=0pt, minimum
width=4pt]

 \node (n_1) at (-.2,.2)[label=left:{$x_1$}] {};
 \node (n_2) at (-.2,0)[label=left:{$x_2$}] {};
 \node (n_3) at (-.2,-.2)[label=left:{$x_3$}] {};
 \node (n_4) at (.2,.2)[label=right:{$y_1$}] {};
 \node (n_5) at (.2,0)[label=right:{$y_2$}] {};
 \node (n_6) at (.2,-.2)[label=right:{$y_3$}] {};
 \node (n_7) at (-.2,-.4)[label=left:{$x_4$}] {};
 \node (n_8) at (.2,-.4)[label=right:{$y_4$}] {};

\foreach \from/\to in
  { n_1/n_4,n_1/n_6, n_2/n_5,n_3/n_8,n_7/n_8}\draw (\from) -- (\to);
\end{tikzpicture}
\caption{$G$}
\label{g}
\end{minipage}
\begin{minipage}[b]{0.35\linewidth}
\centering
\begin{tikzpicture}[scale=4]
\tikzstyle{every node}=[circle, fill=black!,inner sep=0pt, minimum
width=4pt]

 \node (n_1) at (-.2,.2)[label=left:{$x_1$}] {};
 \node (n_2) at (-.2,0)[label=left:{$x_2$}] {};
 \node (n_3) at (-.2,-.2)[label=left:{$x_3$}] {};
 \node (n_4) at (.2,.2)[label=right:{$y_1$}] {};
 \node (n_5) at (.2,0)[label=right:{$y_2$}] {};
 \node (n_6) at (.2,-.2)[label=right:{$y_3$}] {};

\foreach \from/\to in
  { n_1/n_4,n_1/n_5,n_2/n_6,n_3/n_6}\draw (\from) -- (\to);
\end{tikzpicture}
\caption{$G^{\prime}$}
\label{gg}
\end{minipage}
\end{figure}\\
\end{example}

 
A collection $\Delta$ of subsets of $V= \{x_1, x_2, \ldots, x_n\}$ satisfying $\{x_i\} \in \Delta;1\le i\le n$ and if $F\in \Delta ~\text{and}~ G \subseteq F$ implies $G\in \Delta$ is called a {\it simplicial  complex} on vertex set $V$. The elements of $\Delta$ are called faces and a face $F\in \Delta$ of cardinality $i+1$ is called an $i$-{\it face} or a {\it face} of dimension $i$ of $\Delta$. A subset $F{^\prime}\subset V$ with $F{^\prime}\notin\Delta$ is called a {\it non-face} of $\Delta$. If every subset of $V$ is a member of $\Delta$, then $\Delta$ is called a {\it simplex}. The maximum of dimension of all the faces of a simplicial complex $\Delta$ is its dimension, denoted by $\dim(\Delta)$. For a subset $W\subseteq V$, the {\it induced subcomplex} of $\Delta$ on $W\subseteq V$, denoted by $\Delta_W$, is given by
\[
\Delta_W=\{F\in \Delta~|~F\subseteq W\}.
\]
 
For a simplicial complex $\Delta$ on vertex set $V=\{x_1,x_2,\ldots, x_n\}$, the \textit{Stanley-Reisner ideal} or {\it non-face ideal} of $\Delta$, denoted by $I_\Delta$, is a squarefree monomial ideal in $R=\Bbbk[x_1,x_2, \ldots, x_n]$ generated by all monomials $x_{i_1}x_{i_2}\cdots x_{i_r}$ such that $\{x_{i_1},x_{i_2},\ldots,x_{i_r}\}$ is a non-face of $\Delta$, and the quotient ring $\Bbbk[\Delta]=R/I_\Delta$ is called the {\it Stanley-Reisner ring} of $\Delta$. The natural map $\Delta\longleftrightarrow I_{\Delta}$ is bijection between the class of simplicial complexes on vertex set $V=\{x_1,x_2,\ldots,x_n\}$ and the class of squarefree monomial ideals in the polynomial ring $R=\Bbbk[x_1,x_2,\ldots,x_n]$. For more details, we refer to~\cite{Vill}.

It can be seen that the edge ideal $I(G)$ of $G$ is same as the Stanley-Reisner ideal $I_{\Delta(G)}$ of the independence complex $\Delta(G)$, where $\Delta(G)$ is a simplicial complex on vertex set $V=\{x_1,x_2,\ldots,x_n\}$ of $G$ given by
\[
\Delta(G)=\{S\subseteq V:~S~\text{is an independent subset of graph}~G\}.
\]

Now, we recall a well-known result, known as Hochster's formula, which is an important tool for computation of graded Betti numbers of Stanley-Reisner ring $k[\Delta]$.

\begin{theorem}\cite[Hochster's Formula]{hoch}
The $i$th $\mathbb{N}$-graded Betti number of the Stanley-Reisner ring $\Bbbk[\Delta]=R/I_\Delta$ in degree $j$ is given by
\begin{equation}
\beta_{i,j}(\Bbbk[\Delta])=\sum_{W\subseteq V, |W|=j}^{}\dim_\Bbbk \widetilde{H}_{j-i-1}(\Delta_W; \Bbbk).
\label{hoc}
\end{equation}
\end{theorem}

\section{Linear Strand of $R/I(\mathcal{H}(m,k))$}
In this section, we recall the notion of bipartite Kneser graphs and obtain a combinatorial formula for computing the graded Betti numbers in the linear strand of edge ideals of these graphs using the Hochster's formula. 
\begin{definition}\label{def}\rm
Let $m$ and $k$ be integers with $m\ge 2k$. Set $[m]=\{1,2,\ldots,m\}$. The {\it bipartite Kneser graph} $\mathcal{H}(m,k)$ is a bipartite graph with vertex set $V\left(\mathcal{H}(m,k)\right)=V_1\sqcup V_2$, where 
\[
V_1=\{A\subset [m]:~A~\text{is a $k$-element subset of}~[m]\}
\]
and
\[
V_2=\{B\subset [m]:~B~\text{is a $(m-k)$-element subset of}~[m]\},
\]
and for $A\in V_1$ and $B\in V_2$, the pair $\{A,B\}\in E\left(\mathcal{H}(m,k)\right)$ if $A\subseteq B$.
\end{definition}
It is easy to see that the bipartite Kneser graph $\mathcal{H}(m,k)$ has 2$m\choose k$ vertices. For, $A=\{i_1,i_2,\ldots,i_k\}\subset [m]$ and $B=\{j_1,j_2,\ldots,j_{m-k}\}\subset [m]$ with $|B|=m-k$, we see that $A\subseteq B$ iff $B=A\sqcup S$, where $S\subset [m]\setminus A$ with $|S|=m-2k$. Thus, the number of such subsets $B\subset [m]$ is ${m-k\choose k}$.  Also, a $(m-k)$-subset $B\subset [m]$ contains exactly ${m-k\choose k}$ $k$-subsets. Thus each vertex of the bipartite Kneser graph $\mathcal{H}(m,k)$ is of the same degree $m-k\choose k$, hence it is a regular graph having ${m-k\choose k}{m\choose k}$ edges.

For $k=1$,  the bipartite Kneser graph $\mathcal{H}(m,1)$ is isomorphic to the $m$-crown graph $\mathcal{C}_{m,m}$, (see~\cite{shpa}, for definition). Rather and Singh~\cite{shpa} computed the graded Betti numbers, regularity and projective dimension of crown graphs, and hence of bipartite Kneser graph $\mathcal{H}(m,1)$.

\begin{example}\rm The bipartite Kneser graph $\mathcal{H}(5,2)$ is shown in Figure~\ref{h52} as follows. 
\begin{figure}[ht]
\begin{tikzpicture}[scale=3.75]
\tikzstyle{every node}=[circle, fill=black!,inner sep=1pt, minimum
width=6pt]
 \node (n_1) at (0,0)[label=left:{$\{1,2\}$}] {};
 \node (n_2) at (0,-0.2)[label=left:{$\{1,3\}$}] {};
 \node (n_3) at (0,-0.4)[label=left:{$\{1,4\}$}] {};
 \node (n_4) at (0,-0.6)[label=left:{$\{1,5\}$}] {};
 \node (n_5) at (0,-0.8)[label=left:{$\{2,3\}$}] {};
 \node (n_6) at (0,-1)[label=left:{$\{2,4\}$}] {};
 \node (n_7) at (0,-1.2)[label=left:{$\{2,5\}$}] {};
 \node (n_8) at (0,-1.4)[label=left:{$\{3,4\}$}] {};
 \node (n_9) at (0,-1.6)[label=left:{$\{3,5\}$}] {};
 \node (n_10) at (0,-1.8)[label=left:{$\{4,5\}$}] {};
 \node (n_11) at (1,0)[label=right:{$\{1,2,3\}$}] {};
 \node (n_12) at (1,-0.2)[label=right:{$\{1,2,4\}$}] {};
 \node (n_13) at (1,-0.4)[label=right:{$\{1,2,5\}$}] {};
 \node (n_14) at (1,-0.6)[label=right:{$\{1,3,4\}$}] {};

 \node (n_15) at (1,-0.8)[label=right:{$\{1,3,5\}$}] {};
 \node (n_16) at (1,-1)[label=right:{$\{1,4,5\}$}] {};
 \node (n_17) at (1,-1.2)[label=right:{$\{2,3,4\}$}] {};
 \node (n_18) at (1,-1.4)[label=right:{$\{2,3,5\}$}] {};
 \node (n_19) at (1,-1.6)[label=right:{$\{2,4,5\}$}] {};
 \node (n_20) at (1,-1.8)[label=right:{$\{3,4,5\}$}] {};

\foreach \from/\to in
  { n_1/n_11,n_1/n_12,n_1/n_13,n_2/n_11,n_2/n_14,n_2/n_15,n_3/n_12,n_3/n_14,n_3/n_16,n_4/n_13,n_4/n_15,n_4/n_16,n_5/n_11,n_5/n_17,n_5/n_18,n_6/n_12,n_6/n_17,n_6/n_19,n_7/n_13,n_7/n_18,n_7/n_19,n_8/n_14,n_8/n_17,n_8/n_20,n_9/n_15,n_9/n_18,n_9/n_20,n_10/n_16,n_10/n_19,n_10/n_20}\draw (\from) -- (\to);  
\end{tikzpicture}
\caption{$\mathcal{H}(5,2)$}
\label{h52}
\end{figure}
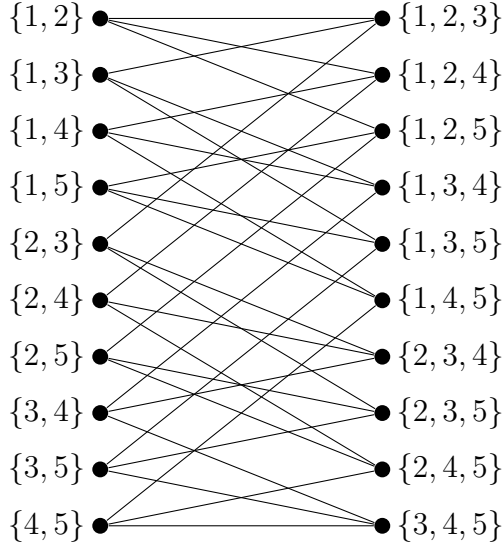
\end{example}

To compute the Betti numbers in the linear strand of edge ideals of bipartite Kneser graphs, we introduce the following notations.

\begin{notation}\noindent 
{\rm
\begin{enumerate}
\item 
		
For a finite set $X$, the family of all subsets of $X$ of size $s$ is denoted by $X^{(s)}$. For positive integers $m,q$ and $r$, we set $$\mathfrak{S}_{m,q,r}=\left([m]^{(r)}\right)^{(q)}=\left\{\mathcal{A} \subseteq [m]^{(r)}: \mid \mathcal{A} \mid=q\right\}.$$

\item For $i \in [m]=\{1,2,\ldots,m\}$, let an element $\mathcal{A}$ of $\mathfrak{S}_{m,q,r}$ have property $P(i)$ if and only if  $i \in \bigcap\limits_{A \in \mathcal{A}}A.$ Further, for any subset $T \subseteq \{1,2,\ldots,m\},$ we define

\begin{eqnarray*}\mathfrak{S}_{m,q,r}^{\supseteq T}&=&\left\{ \mathcal{A} \in \mathfrak{S}_{m,q,r}: \mathcal{A}~ \text{has ~atleast~ the~ property}~ P(j)~ \text{for~each}~j\in T  \right\} \\
\mathfrak{S}_{m,q,r}^{= T} &=& \left\{ \mathcal{A} \in \mathfrak{S}_{m,q,r}: \mathcal{A}~ \text{has ~precisely~ the~ property}~ P(j)~ \text{for~each}~j\in T  \right\}.
\end{eqnarray*} 
\end{enumerate}

Now, we set 
\[
N_{m,q,r}^{\supseteq T}=|\mathfrak{S}_{m,q,r}^{\supseteq T}| ~\text{and}~ N_{m,q,r}^{= T}=\left|\mathfrak{S}_{m,q,r}^{=T}\right|.
\]
}
\end{notation}

\begin{lemma}\rm
With the same notations as above, we have
\[
N_{m,q,r}^{=\emptyset}=\sum_{j=0}^{r}(-1)^j{m\choose j}\binom{{m-j\choose r-j}}{q}.
\]
\end{lemma}
\begin{proof}
From the principle of inclusion-exclusion~(see~\cite{MA}, Chapter 5), we see that
\begin{eqnarray*}
N_{m,q,r}^{=\emptyset}&=&\sum_{T \subseteq \{1,2,\ldots,m\}}(-1)^{|T|}N_{m,q,r}^{\supseteq T}
\end{eqnarray*}
Since $N_{m,q,r}^{\supseteq T}$ depends only on the
size $|T| = t,$ we write $N_{m,q,r}^{\supseteq T}=N_{m,q,r}^{\geq t}$ for $|T| = t$. Also, in this case
 $N_{m,q,r}^{= T}=N_{m,q,r}^{= t}$ depends only on the cardinality of $T.$ Hence

\begin{eqnarray*}
N_{m,q,r}^{=\emptyset}=	N_{m,q,r}^{=0}&=& \sum_{t=0}^{m}(-1)^t{m\choose t}N_{m,q,r}^{\geq t}\\
	  &=&\sum_{t=0}^{r}(-1)^t {m\choose t}\binom{{m-t\choose r-t}}{q},
\end{eqnarray*}

since $N_{m,q,r}^{\geq t}=0$ for $t>r.$
\end{proof}

\begin{remark}
	\noindent 
	\rm{
	
\begin{enumerate}
	\item Note that
	\[
	N_{m,q,r}^{=t}=N_{m-t,q,r-t}^{=0}=\sum_{j=0}^{r-t}(-1)^j{m-t \choose j}\binom{{m-t-j\choose r-t-j}}{q}.
	\]
	\item For $q=1, t\neq r,$ we have 
	\begin{eqnarray*}
	N_{m,q,r}^{=t}&=&\sum_{j=0}^{r-t}(-1)^j{m-t \choose j}{m-t-j\choose r-t-j}\\
	&=& \sum_{j=0}^{r-t}(-1)^j{m-t \choose r-t}{r-t\choose j}\\
	&=& {m-t \choose r-t}\sum_{j=0}^{r-t}(-1)^j{r-t\choose j}\\
	&=&0.
	\end{eqnarray*}

\item 	$N_{m,q,r}^{=t}=\left\{
\begin{array}{lll}
1  &;~q=1,  t=r \\
0 & ;~q \neq 1, t=r.
\end{array}\right.
$
\end{enumerate}		
}
\end{remark}

\noindent
\begin{theorem}\label{mainb}
Let $\mathcal{H}(m,k)$ be the bipartite Kneser graph with vertex set $V(\mathcal{H}(m,k))=V_1 \sqcup V_2$ as given in Definition~\ref{def}. Then 
\[
\beta_{i,i+1}(\mathcal{H}(m,k))=\sum_{r+s=i+1}\sum_{t=k}^{m-k}\binom{{t\choose k}}{r}{m\choose t}N_{m,s,m-k}^{=t}.
\]
\end{theorem}

\begin{proof}
Set $\Delta=\Delta(\mathcal{H}(m,k))$. Then, in view of Hochster's Formula, we have
\begin{equation}
\beta_{i,i+1}(\mathcal{H}(m,k))=\sum_{W\subseteq V, |W|=i+1}^{}\dim_\Bbbk \widetilde{H}_{0}(\Delta_W; \Bbbk).
\label{31}
\end{equation}
If $W\subseteq V_1$ or $W\subseteq V_2$, then $\Delta_W$ is a simplex, so has zero reduced homology. Now for $r,s\ge 1$, let $W\subseteq V_1 \cup V_2$ is of the form $W_{r,s}=\mathcal{A}_r \cup \mathcal{B}_s$, where $\mathcal{A}_r=\{A_1,\ldots,A_r\}, \mathcal{B}_s=\{B_1,\ldots,B_s\},$ with $A_i \in [m]^{(k)}$ and $B_j \in [m]^{(m-k)}$ for all $1 \leq i \leq r, 1 \leq j \leq s$. We know that  $\widetilde{H}_{0}(\Delta_W; \Bbbk) =  \Bbbk$  if and only if $\Delta_W$ has two connected components. Observe that if $W$ is of the form $W_{r,s}$ as given above, then $\Delta_W$ has two connected components if and only if $A \subseteq \bigcap\limits_{B \in \mathcal{B}_s} B$ for all $A \in \mathcal{A}_r$. Let $W'_{r,s}=\mathcal{A}_r \cup \mathcal{B}_s$ such that $A \subseteq \bigcap\limits_{B \in \mathcal{B}_s} B$ for all $A \in \mathcal{A}_r$, where $\mathcal{A}_r, \mathcal{B}_s$ are defined as above. Then, in view of equation~(\ref{31}), we have 
$
\beta_{i,i+1}(\mathcal{H}(m,k))=\text{number of possible subsets}~W'_{r,s}~\text{such that}~r+s=i+1.
$
 Observe that for any such subset $W'_{r,s}$, we have $k\le \left|\bigcap\limits_{B \in \mathcal{B}_s} B\right|\le m-k$. Further, the number of subsets  of the set $[m]^{(m-k)}$ of type $\mathcal{B}_s$ with $\left|\bigcap\limits_{B \in \mathcal{B}_s}  B\right|=t$, is given by ${m\choose t}N_{m,s,m-k}^{=t}$ . Also, the number subsets of the set $[m]^{(k)}$ of type  $\mathcal{A}_r$ such that $\left|\bigcap\limits_{B \in \mathcal{B}_s}  B\right|=t$ and $A \subseteq \bigcap\limits_{B \in \mathcal{B}_s} B$ for all $A \in \mathcal{A}_r$,  is given by
$
\binom{{t\choose k}}{r}.
$
Thus, the number of subsets $W'_{r,s}$ with $r+s=i+1$  equals to
\[
\sum_{r+s=i+1}\sum_{t=k}^{m-k}\binom{{t\choose k}}{r}{m\choose t}N_{m,s,m-k}^{=t}
\]  
and so, we have
\[
\beta_{i,i+1}(\mathcal{H}(m,k))=\sum_{r+s=i+1}\sum_{t=k}^{m-k}\binom{{t\choose k}}{r}{m\choose t}N_{m,s,m-k}^{=t},
\]
as desired.
\end{proof}
\begin{example}\rm 
The graded Betti numbers $\beta_{i,i+1}\left(R/I(\mathcal{H}(5,2))  \right) $ by using Theorem~\ref{mainb} can be computed as follows: 

\begin{align*}
\beta_{1,2}(\mathcal{H}(5,2))& =  \sum_{r+s=2}\sum_{t=2}^{3}\binom{{t\choose 2}}{r}{5\choose t}N_{5,s,3}^{=t}=\sum_{t=2}^{3}{t\choose 2}{5\choose t}N_{5,1,3}^{=t}\\
&={2\choose 2}{5\choose 2}N_{5,1,3}^{=2}+{3\choose 2}{5\choose 3}N_{5,1,3}^{=3}\nonumber\\
&=0+3{5\choose 3}=30,\nonumber\\
\beta_{2,3}(\mathcal{H}(5,2))& = \sum_{r+s=3}\sum_{t=2}^{3}\binom{{t\choose 2}}{r}{5\choose t}N_{5,s,3}^{=t}=\sum_{t=2}^{3}\binom{{t\choose 2}}{1}{5\choose t}N_{5,2,3}^{=t}+\sum_{t=2}^{3}\binom{{t\choose 2}}{2}{5\choose t}N_{5,1,3}^{=t}\nonumber\\
&=\left[{2\choose 2}{5\choose 2}N_{5,2,3}^{=2}+{3\choose 2}{5\choose 3}N_{5,2,3}^{=3}\right]\\
&\phantom{{}={}{2\choose 2}}+ \left[\binom{{2\choose 2}}{2}{5\choose 2}N_{5,1,3}^{=2}+\binom{{3\choose 2}}{2}{5\choose 3}N_{5,1,3}^{=3}\right]\nonumber\\
&=\left[1{5\choose 2}N_{3,2,1}^{=0}+0\right]+\left[0+3{5\choose 3}\right]
\nonumber\\
&=\left[10\sum_{j=0}^{1}{(-1)^j}{3\choose j}\binom{{{3-j}\choose {1-j}}}{2}\right]+30
\nonumber\\
&=\left[10{(-1)^0}{3\choose 0}\binom{{3\choose 1}}{2}+10{(-1)^1}{3\choose 1}\binom{{2\choose 0}}{2}\right]+30=\left[30+0\right]+30=60,\nonumber
\end{align*}
\begin{align*}
\beta_{3,4}(\mathcal{H}(5,2))& = \sum_{r+s=4}\sum_{t=2}^{3}\binom{{t\choose 2}}{r}{5\choose t}N_{5,s,3}^{=t}=\sum_{t=2}^{3}\binom{{t\choose 2}}{1}{5\choose t}N_{5,3,3}^{=t}+\sum_{t=2}^{3}\binom{{t\choose 2}}{2}{5\choose t}N_{5,2,3}^{=t}\nonumber\\
&+\sum_{t=2}^{3}\binom{{t\choose 2}}{3}{5\choose t}N_{5,1,3}^{=t}=\left[\binom{{2\choose 2}}{1}{5\choose 2}N_{5,3,3}^{=2}+\binom{{3\choose 2}}{1}{5\choose 3}N_{5,3,3}^{=3}\right]\nonumber\\
&+\left[\binom{{2\choose 2}}{2}{5\choose 2}N_{5,2,3}^{=2}+\binom{{3\choose 2}}{2}{5\choose 3}N_{5,2,3}^{=3}\right]\\
&\phantom{{}={}{2\choose 2}}
+\left[\binom{{2\choose 2}}{3}{5\choose 2}N_{5,1,3}^{=2}+\binom{{3\choose 2}}{3}{5\choose 3}N_{5,1,3}^{=3}\right]\nonumber\\
&=\left[1{5\choose 2}N_{3,3,1}^{=0}+0\right]+\left[0+0\right]+\left[0+1{5\choose 3}\right]\nonumber\\
&=\left[10\sum_{j=0}^{1}{(-1)^j}{3\choose j}\binom{{{3-j}\choose {1-j}}}{3}\right]+10
\nonumber\\
&=\left[10{(-1)^0}{3\choose 0}\binom{{3\choose 1}}{3}+10{(-1)^1}{3\choose 1}\binom{{2\choose 0}}{3}\right]+10=\left[10+0\right]+10=20,
\end{align*} 
and $\beta_{i,i+1}(R/I(\mathcal{H}(5,2))=0$ for $i\ge 4$. We have seen the Betti numbers of $R/I(\mathcal{H}(5,2))$ computed using Theorem~\ref{mainb} are same as computed using Singular 2.0~\cite{sin}.
\end{example}


\section{Regularity of powers of edge ideals of bipartite Kneser graphs}
\noindent

 The regularity of edge ideals of a graph is an important homological invariant studied by various authors, see~\cite{AB,CH,hawo,janase,jase,katz,wood} and references therein. The regularity of bipartite Kneser graph $\mathcal{H}(m,1)$ has been computed by Rather and Singh~\cite{shpa}.

Before proceeding further, we shall make some remarks about the bipartite Kneser graphs as follows. 
\begin{remark}\label{3dst}\noindent \rm 
\begin{enumerate}
	\item For a bipartite Kneser graph $\mathcal{H}(m,k)$, it can be easily seen that $e=\{A,B\}$ and $e^\prime=\{A^{\prime},B^{\prime}\}$ in $\mathcal{H}(m,k)$ are 3-disjoint if and only if  $A\nsubseteq B^{\prime}$ and $A^{\prime}\nsubseteq B$.  

\item Let $S=\{j_1,j_2,\ldots,j_{m-2k}\}\subset [m]$ be a fixed subset of $[m]$ having $m-2k$ elements. Consider the collection $E_S$ of all edges in $\mathcal{H}(m,k)$ of the form
\[
e=\{\{i_1,i_2,\ldots,i_k\},\{i_1,i_2,\ldots,i_k,j_1,j_2,\ldots,j_{m-2k}\}\}
\] 
where $\{i_1,i_2,\ldots,i_k\}\subseteq [m]\setminus\{j_1,j_2,\ldots,j_{m-2k}\}$. Then, clearly, by the Remark~\ref{3dst}(1), $E_S$ is a pairwise $3$-disjoint subset of $E(\mathcal{H}(m,k))$ having 
$
{2k\choose k}
$
elements. 

\end{enumerate}	
\end{remark}
\begin{lemma}
 Let $E_S$ be as given in Remark~\ref{3dst}. Then $E_S$ is a maximal pairwise $3$-disjoint subset of $E(\mathcal{H}(m,k))$.
\end{lemma}
\begin{proof}
	 For, let 
\[
e^\prime=\{\{s_1,s_2,\ldots,s_k\},\{s_1,s_2,\ldots,s_k,t_1,t_2,\ldots,t_{m-2k}\}\}
\] 
be any edge of $\mathcal{H}(m,k)$ other than the edges in $E_S$. Now, if $\{s_1,s_2,\ldots,s_k\}\subseteq [m]\setminus\{j_1,j_2,\ldots,j_{m-2k}\}$ or $\{s_1,s_2,\ldots,s_k\}\subseteq \{j_1,j_2,\ldots,j_{m-2k}\}$, then by definition of $\mathcal{H}(m,k)$, the end vertex $\{s_1,s_2,\ldots,s_{k}\}$ of $e^\prime$ is also an end vertex of some edge in $E_S$ or $\{s_1,s_2,\ldots,s_{k}\}$ is connected to an end vertex of some edge in $E_S$, and so we are done in this case. Now, if neither $\{s_1,s_2,\ldots,s_k\}\subseteq [m]\setminus\{j_1,j_2,\ldots,j_{m-2k}\}$ nor $\{s_1,s_2,\ldots,s_k\}\subseteq \{j_1,j_2,\ldots,j_{m-2k}\}$, then $\{s_1,s_2,\ldots,s_k\}\cap \{j_1,j_2,\ldots,j_{m-2k}\}\ne \emptyset$ and $\{s_1,s_2,\ldots,s_k\}\cap [m]\setminus \{j_1,j_2,\ldots,j_{m-2k}\}\ne \emptyset$. Without loss of generality, we may assume that $\{s_1,s_2,\ldots,s_k\}\cap \{j_1,j_2,\ldots,j_{m-2k}\}=\{s_1,s_2,\ldots,s_p\}=\{j_1,j_2,\ldots,j_p\}$ with $p<k$ and $\{s_1,s_2,\ldots,s_k\}\cap [m]\setminus \{j_1,j_2,\ldots,j_{m-2k}\}=\{s_{p+1},\ldots,s_k\}$. Observe that $\{s_{p+1},\ldots,s_k\}\subset [m]\setminus \{j_1,j_2,\ldots,j_{m-2k}\}$ is a subset of some $k$-element subset, say $A$, of $[m]\setminus \{j_1,j_2,\ldots,j_{m-2k}\}$. Then 
\[
\{s_1,s_2,\ldots,s_k\}\subseteq A\sqcup \{j_1,j_2,\ldots,j_{m-2k}\}.
\]
Clearly, $A\sqcup \{j_1,j_2,\ldots,j_{m-2k}\}$ is an end vertex of one of the edges in $E_S$. Thus, $\{s_1,s_2,\ldots,s_{k}\}$ is connected to an end vertex of some edge in $E_S$ in this case also. Hence, $E_S$ is a maximal pairwise $3$-disjoint subset of $E(\mathcal{H}(m,k))$. 
\end{proof}
\begin{remark}\rm
Since the induced matching number is the maximum size of all maximal pairwise 3-disjoint subsets of an edge set of a graph, we have
\[
{\rm ind}(\mathcal{H}(m,k))\geq |E_S|.
\]
\label{max.}
\end{remark}
Thus, in view of lower bound obtained by Beyarslan, H\`{a} and Trung~\cite{behatr} given in the equation~(\ref{eqn1}) and Remark~\ref{max.}, we have the following result:

\begin{theorem}\label{main1}
For the bipartite Kneser graph $\mathcal{H}(m,k)$, we have
\[
{\rm reg}(R/{I(\mathcal{H}(m,k))}^p)\geq 2(p-1)+{2k\choose k}.
\]
\end{theorem}
Note that for $m=2k$, the bipartite Kneser graph $\mathcal{H}(m,k)$ is the ladder rung graph ${LR}_{2k\choose k}$, where ${LR}_{2k\choose k}$ is the graph union of $2k\choose k$ copies of the path graph $P_2$. It can easily seen that the induced matching number and co-chordal number of the ladder rung graph is simply the number of copies of the path graph $P_2$ in it.

Now, we shall deduce the bounds for regularity of powers of edge ideals of $\mathcal{H}(m,k)$, and show that lower bound is attained in some cases as follows.
\begin{theorem}\label{regb}
For a bipartite Kneser graph $\mathcal{H}(m,k)$, we have
\[
2(p-1)+{2k\choose k}\leq {\rm reg}(R/{I(\mathcal{H}(m,k))}^p)\leq 2(p-1)+{m\choose k}.
\] 
Furthermore, the lower bound is attained, if $m=2k$ or $2k+1$.
\end{theorem}
\begin{proof}
\noindent
By definition of $\mathcal{H}(m,k)$, we have $m\ge 2k$. Consider the collection $\mathcal{A}=\{A_1,\ldots,A_{m\choose k}\}$ of all $k$-element subsets of $[m]$. Note that each star ${\mathcal{S}}_{A_j};1\le j\le {m\choose k}$, is a co-chordal subgraph of $\mathcal{H}(m,k)$ and 
\[
\bigsqcup^{m\choose k}_{j=1}E(\mathcal{S}_{A_j})=E(\mathcal{H}(m,k)).
\] 
Thus, in view of definition, we have
\begin{equation}\label{eqn3}
{\rm cochord}(\mathcal{H}(m,k))\leq {m\choose k}.
\end{equation}
Hence by Theorem~\ref{main1} and \cite[Theorem 3.6]{janase}, we get
\[
2(p-1)+{2k\choose k}\leq {\rm reg}(R/{I(\mathcal{H}(m,k))}^p)\leq 2(p-1)+{m \choose k},
\] 
as desired. 

Now, for $m=2k$, the above inequalities gives the equality
 \[
{\rm reg}(R/{I(\mathcal{H}(m,k))}^p)=2(p-1)+{2k\choose k},
\] 
as required. Also, for $m=2k+1$, consider the  maximal pairwise $3$-disjoint subset $E_S$ of $E(\mathcal{H}(m,k))$ as given in Remark~\ref{3dst}. Then $|E_S|={2k\choose k}$ and $m-2k=1$. Thus, by definition, every edge $\{A,B\}\in E_S$ is the form $B=A\sqcup \{t\}$, where $t\in [m]$ is a fixed element and $A$ is a $k$-element subset of $[m]\setminus\{t\}$. Then one can write
\[
E_S=\left\{\{A_1,B_1\},\ldots,\{A_{2k\choose k},B_{2k\choose k}\}\right\},
\]
where each $A_i$ is a $k$-element subset of $[m]\setminus\{t\}$ and $B_i=A_i\sqcup\{t\};~1\le i\le {2k\choose k}$, is a $(m-k)$-element subset of $[m]$. Define $\mathcal{S}_{A_iB_i}=\mathcal{S}_{A_i}\cup \mathcal{S}_{B_i}$, the union of star graphs $\mathcal{S}_{A_i}$ and $\mathcal{S}_{B_i};1\le i\le \binom{2k}{k}$. Then observe that each $\mathcal{S}_{A_iB_i}$ is a co-chordal subgraph of $\mathcal{H}(m,k)$. We claim that
\[
\bigsqcup^{2k\choose k}_{i=1}E(\mathcal{S}_{A_iB_i})=E(\mathcal{H}(m,k)).
\]
For, let $e=\{A,B\}\in E(\mathcal{H}(m,k))$. Then $B$ is a $(k+1)$-element set and $A$ is a $k$-element set with $A\subset B$, since $m=2k+1$. Now if $t\in B$ such that $t\notin A~ \textrm{or}~$, then $B=A\sqcup\{t\}$ and $A$ is $k$-element subset of $[m]\setminus\{t\}$, so $A=A_r$ and $B=B_r$, for some $r$. This shows that $e=\{A,B\}\in E(\mathcal{S}_{A_rB_r})$, for this $r$. If $t\notin B$, then 
$A\subset B\subset [m]\setminus \{t\}$. This shows that $A=A_s$, for some $s$, so $e=\{A,B\}\in E(\mathcal{S}_{A_sB_s})$. This proves our claim. Thus, by definition, we have
\begin{equation}\label{eqn2}
{\rm cochord}(\mathcal{H}(m,k))\leq {2k\choose k}.
\end{equation}
Therefore, in view of Theorem~\ref{main1} and \cite[Theorem 3.6]{janase}, we have
\[
{\rm reg}(R/{I(\mathcal{H}(m,k))}^p)=2(p-1)+{2k\choose k},
\] 
as desired.
\end{proof}
 
Now, we shall deduce an another upper bound for regularity of edge ideals of bipartite Kneser graphs which is finer than as obtained in Theorem~\ref{regb}. M{\"u}tze and Su~\cite{musu} settled the following conjecture  for bipartite Kneser graphs, which together with equation~(\ref{EQ}) gives an upper bound for regularity of bipartite Kneser graphs. 

\begin{theorem}\cite{musu}\label{con}
For any $k\geq 1$ and $m\geq 2k+1$, the bipartite Kneser graph $\mathcal{H}(m,k)$ has a Hamiltonian cycle.
\end{theorem}
Thus, in view of the equation~(\ref{EQ}) and Theorem~\ref{con}, we have the following.

\begin{corollary}\label{con1}
For $k\geq 1$ and $m\geq 2k+1$, we have
\[
{\rm reg}(R/{I(\mathcal{H}(m,k))})\leq \left\lfloor\frac{{2{m\choose k}}+1}{3}\right\rfloor .
\]
\end{corollary}

Thus, in view of Theorem~\ref{main1} and Corollary~\ref{con1}, we get the following result.
\begin{theorem}
	For a bipartite Kneser graph $\mathcal{H}(m,k)$, we have
	\[
	\binom{2k}{k}\le {\rm reg}(R/{I(\mathcal{H}(m,k))})\leq \left\lfloor\frac{{2{m\choose k}}+1}{3}\right\rfloor .
	\]
\end{theorem}

Rather and Singh shown that ${\rm reg}(R/{I(\mathcal{H}(m,k))})=2$ which is equal to $\binom{2k}{k}$, if $k=1$. Also, we have ${\rm reg}(R/{I(\mathcal{H}(m,k))})=\binom{2k}{k}$, if $m=2k+1$, in view of Theorem~\ref{regb}. 

\section{Projective dimension of bipartite  Kneser graphs}

Now, we deduce bounds on projective dimension of edge ideals of bipartite Kneser graphs using domination parameters. Before proceeding further, we recall the definitions of some domination parameters of graphs. 
 
For a finite simple graph $G=(V(G),E(G))$, we write $G^{0}$ for the graph obtained from $G$ by removing the isolated vertices. 
\begin{itemize}
\item A subset $U \subseteq V(G)$ is called \textit{dominating set} if every vertex of $V(G) \setminus U$ is a neighbour of some vertex of $U$, that is, $N_G[U]=V(G)$, where $N_G[U]=\bigcup_{u\in U}N_G[u]$. The number 
\[
\gamma(G)=\min \{|C|: C \subseteq V(G) ~\text{is~ a~ dominating ~set}\}
\]
is called {\it domination number}, and the number
\[
i(G)=\min\{|C|: C \subseteq V(G)~\text{ is~ independent ~and ~a ~dominating ~set ~in} ~G\}
\]
is called an {\it independent domination number} of $G$.

\item For $C\subseteq V(G)$, we define 
\[
\gamma(C,G)=\text{min}\{|X|: X\subseteq V(G)~\text{with}~C\subseteq N_G(X)\},
\]
where $N_G(X)=\bigcup_{x\in X}N_G(x)$.
Then the number
\[
\tau(G)=\max \{\gamma(C,G^0):C \subseteq V(G^0) ~\text{is~ an independent set}\}
\]
is called {\it independence domination number} of $G$.
\end{itemize}

\begin{notation} \label{dom}\noindent \rm
Consider the graph $\mathcal{H}(m,k)$ and $S=\{j_1,j_2,\ldots,j_{m-2k}\}\subset [m]$ be a fixed subset of $[m]$ having $m-2k$ elements with $m>2k$. Choose $j \in [m] \setminus S$ and define $T=\{j_1,j_2,\ldots,j_{m-2k}\} \cup \{j\}$. Consider the collection $\mathcal{A}'=\{A_1',\ldots,A_r'\}$ of all possible $k$-subsets of $[m]$ with $A_{i}^{\prime} \cap T =\emptyset$, and $\mathcal{B}'=\{B_{1}^{\prime},\ldots,B_{r}^{\prime}\}$ be the collection of all possible  $(m-k)$-subsets of $[m]$ with $T \subseteq B_i'$ for all $i$. Then, we can write $B_i'=T \sqcup T_{i}^{\prime}$, where $T_{i}^{\prime}$ is a $(k-1)$-subset of $[m]\setminus T$ for each $i$. Clearly, we have $|\mathcal{A}'|=|\mathcal{B}'|={2k-1 \choose k-1}=\frac{1}{2}{2k \choose k}$. We set $\mathcal{W}=\mathcal{A}' \cup \mathcal{B}'$.
\end{notation}

\begin{lemma}\label{lower}
	Let $\mathcal{W}$ be as given in Notation~\ref{dom}. Then $\mathcal{W}$ is a dominating set in $\mathcal{H}(m,k)$. 
	\end{lemma}
\begin{proof}
	Let $A \in V(\mathcal{H}(m,k)) \setminus \mathcal{W}$. Then we have the following two cases to be consider:\\

\noindent
{\bf Case 1.} If $A \in V_1 \setminus \mathcal{W}$, then $A \notin \mathcal{A}'$ and $A \cap T \neq \emptyset$. Write $A=U \sqcup U^\prime$, where $U=A \cap T$ and $U^\prime=A\cap ([m]\setminus T)$. If $U^\prime = \emptyset $, then from Notation \ref{dom},  it follows that $A=U\subset T \subseteq B_i'$ for all $i$, and hence $A \in N_G(\mathcal{W})$. Now, if $U^\prime \neq \emptyset $, then, since $1\leq |U^{\prime}| \leq k-1$, it follows from Notation \ref{dom} that $U^\prime \subseteq T_i'$ for some $i$. This shows that $A \subset B_{i}^{\prime}$, so $A \in N_G(\mathcal{W})$.\\

\noindent
{\bf Case 2.} If $A \in V_2 \setminus \mathcal{W}$, then $|A|=m-k$ and $A \notin \mathcal{B}^\prime$. Clearly, $T \not\subseteq A$, otherwise $A\in\mathcal{B}^\prime$. Thus, we see that, if $A \cap T=D$, then $0\leq |D|< m-2k+1$. Write $A=D \sqcup E$, where $E\subseteq [m]\setminus T$ and $k-1 <|E|\leq m-k$. Thus, in view of Notation~\ref{dom}, there exist some $1 \leq j \leq r$ such that $A_j' \subseteq E$, and hence $A \in N_G(\mathcal{W})$.
\end{proof}

\begin{proposition} \label{bound}
	For a bipartite Kneser graph $\mathcal{H}(m,k)$, we have 
\[
{\rm pd}(\mathcal{H}(m,k)) \geq 2{m \choose k}-{2k \choose k}.
\]
\end{proposition}
\begin{proof}
For $m=2k$, we are done. Thus, we may assume that $m >2k$. Using \cite[Proposition 4.7]{dao}, we have ${\rm pd}(\mathcal{H}(m,k)) \geq |V(\mathcal{H}(m,k))|-i(\mathcal{H}(m,k))$. Further, using Lemma \ref{lower} and the fact that $\mathcal{W}$ is an independent, we have   ${2k \choose k}=|\mathcal{W}|\geq i(\mathcal{H}(m,k))$. Hence ${\rm pd}(\mathcal{H}(m,k)) \geq |V(\mathcal{H}(m,k))|-i(\mathcal{H}(m,k)) \geq 2{m \choose k}-{2k \choose k}$.
\end{proof}

\begin{notation} Consider the graph $\mathcal{H}(m,k)$ as given in Definition~\ref{def} and $Q\subset [m]$ with $|Q|=k-1$. Define a collection $\mathcal{D}=\{B \in V_2: Q \subset B\}$. Also, choose a subset $S=\{i_1,\ldots,i_{k+1}\}\subset [m]$ such that $S \cap Q=\emptyset$. Consider the collection $\mathcal{E}=\{Q \sqcup \{i_t\}: 1 \leq t \leq k+1\}.$ 
	\end{notation}

\begin{lemma}\label{upb} With same notations as above, we have $\gamma(\mathcal{D},\mathcal{H}(m,k))=k+1$.
\end{lemma}
\begin{proof}
Let $B \in \mathcal{D}$. Then, by definition of $\mathcal{D}$, we can write $B= Q \sqcup E$, where $|E|=m-2k+1$. We claim that $i_t \in E$ for some $t;1 \leq t \leq k+1$. On the contrary, suppose that $i_j \notin E$ for all $j$, that is, $S \cap E =\emptyset$. Then $|B \cup S|=m+1>m$, a contradiction. Thus, for any $B \in \mathcal{D}$, there exist a $k$- subset $P=Q \sqcup \{i_t\} \in \mathcal{E}$ such that $P \subseteq B$, that is, $\mathcal{D} \subseteq N_G(\mathcal{E})$, and hence  $\gamma(\mathcal{D},\mathcal{H}(m,k)) \leq k+1$.

Next we proceed to show that $\gamma(\mathcal{D},\mathcal{H}(m,k)) = k+1$. On the contrary, suppose that $\gamma(\mathcal{D},\mathcal{H}(m,k))=r, r< k+1,$ i.e., there exists $\mathcal{C}=\{C_1,\ldots,C_r\}$ of $V(G)$ such that $\mathcal{D} \subseteq N_G( \mathcal{C})$. The fact that $Q \neq C_i$ implies that there exists $a_i \in C_i$ such that $a_i \notin Q$ for all $i$. If elements $a_1,\ldots,a_r$ are not distinct, then by deleting repetitions we consider a set $X=\{a_{t_1},\ldots,a_{t_l}\}$, where $1 \leq t_1<\cdots<t_l \leq r$. Let $Y=[m] \setminus (Q \sqcup X)$. Then since $l \leq k$, we have $|Y|=m-(k-1+l) \geq m-2k+1$. Consider a subset $Y_1=\{b_1,\ldots,b_{m-2k+1}\}$ of the set $Y$ and set $Y_2=Q \sqcup Y_1$. Clearly, $Y_2 \in \mathcal{D}$. Now from the construction of the set $Y_2$, it follows that $C_i \not\subseteq Y_2$ for all $i$, a contradiction to the fact that $\mathcal{D} \subseteq N_G( \mathcal{C})$.
\end{proof}

\begin{theorem}{\label{pdmain}}
For a bipartite Kneser graph $\mathcal{H}(m,k)$, we have 
\[
 2{m \choose k}-{2k \choose k} \leq {\rm pd}(\mathcal{H}(m,k)) \leq  2{m \choose k}-\max \left\{ k+1,\frac{{m \choose k}}{{m-k \choose k}}\right\}.
\]
\end{theorem}
\begin{proof}
In view of Proposition~\ref{bound}, it is enough to prove that 
\[
{\rm pd}(\mathcal{H}(m,k)) \leq  2{m \choose k}-\max \left\{ k+1,\frac{{m \choose k}}{{m-k \choose k}}\right\}.
\]
Using \cite[Corollary 5.4]{huneke}, we obtain ${\rm pd}(\mathcal{H}(m,k)) \leq  2{m \choose k}\left\{1- \frac{1}{2{m-k \choose k}}\right\}$. Also, using \cite[Theorem 4.4]{dao}, we have ${\rm pd}(\mathcal{H}(m,k)) \leq 2{m \choose k}-\tau(G)$. Now, Lemma \ref{upb} gives us the desired result.
\end{proof}

\begin{remark}\rm It was shown by Rather and Singh~\cite{shpa} that $\text{pd}
(\mathcal{H}(m,1))=2m-2$ which is equal to $2\binom{m}{k}-\binom{2k}{k}$, if $k=1$. Also, using Singular 2.0~\cite{sin}, we have seen that $\text{pd}
(\mathcal{H}(5,2))=15$ which is not equal to $2\binom{m}{k}-\binom{2k}{k}$, if $m=5$ and $k=2$.
\end{remark}


\begin{thebibliography}{20}
	
	
	\bibitem{MA}  M. Aigner, \emph{A course in enumeration,} Graduate texts in mathematics (Vol. 238). Berlin: Springer, 2007.
	
	\bibitem{AB} A. Banerjee, \emph{The regularity of powers of edge ideals}, J. Algebraic Combin., {\bf 41} (2015), no. 2, 303 –- 321.


\bibitem{behatr} S. Beyarslan, H.T. H\`{a} and T.N. Trung, \emph{Regularity of powers of forests and cycles}, J. Algebraic Combin. {\bf{42}} (2015) 1077--1095.


\bibitem{CH} S. Cutkosky, J. Herzog, and N. V. Trung, \emph{Asymptotic behaviour of the Castelnuovo-Mumford regularity}, Compositio Math., {\bf 118} (1999), 243 -- 261.

\bibitem{huneke}  H. Dao, C. Huneke and J. Schweig, \emph{Bounds on the regularity and projective dimension of ideals associated to graphs} J. Algebraic Combin. {\bf{38}} (2013) 37--55.

\bibitem{dao} H. Dao, J. Schweig, \emph{Projective dimension, graph domination parameters, and independence complex homology}, Journal of Combinatorial Theory, Series A. \textbf{120} (2013), 453--459.

\bibitem{sin} W. Decker, G.-M. Greuel, G. Pfister, and H. Sch$\rm\ddot{o}$nemann, {Singular
	4-0-2\textbf{--}A computer algebra system for research in algebraic geometry}, {Available at {http://www.singular.uni-kl.de}}, 2015.

\bibitem{hatu} H. T. H\`{a} and A. Van Tuyl, \emph{Monomial ideals, edge ideals of hypergraphs, and their graded Betti numbers}, J. Algebraic Combin. \textbf{27} (2008), no. 2, 215--245, arXiv:math/0606539.

\bibitem{hatu1}H. T. H\`{a} and A. Van Tuyl, \emph{Resolutions of square-free monomial ideals via facet ideals: a survey}, in :Algebra, geometry and their intersections, Contemporary Mathematics {\bf 448}, Amer. Math. Soc., Providence, RI, (2007),215--245.

\bibitem{hawo}H. T. H\`{a} and R. Woodroofe, \emph{Results on the regularity of square-free monomial ideals}, Adv. in Appl. Math., {\bf 58}(2014), 21--36.

\bibitem{hoch} M. Hochster, \emph{Cohen-Macaulay rings, combinatorics, and simplicial complexes.} Ring theory II (Proc. Second Conf., Univ. Oklahoma, Norman, Okla., 1975), pp. 171--223.



\bibitem {jacq} S. Jacques, \emph{Betti numbers of graph ideals}, Ph.D. thesis, University of Sheffield, Great Britain.


\bibitem{janase} A.V. Jayanthan, N. Narayanan, and S. Selvaraja, \emph{Regularity of powers of bipartite graphs}, J. Algebraic Combin., 47(1)(2018), 17--38. 

\bibitem{jase} A.V. Jayanthan, and S. Selvaraja, \emph{An upper bound for the regularity of powers of edge ideals of graphs}.Preprint (2018),arXiv:1805.01412v1.

\bibitem{js} A.V. Jayanthan, and S. Selvaraja, \emph{Asymptotic behavior of Castelnuovo-Mumford regularity of edge ideals of very well-covered graphs}, J. Commut. Algebra. (to appear).

\bibitem{katz} M. Katzman, \emph{Characteristic-independence of Betti numbers of graph ideals}, J. Combin. Theory Ser. A {\bf 113} (2006), no. 3, 435--454.

\bibitem{VK} V. Kodiyalam, \emph{Asymptotic behaviour of Castelnuovo-Mumford regularity}, Proc. Amer. Math. Soc., {\bf 128} (2000), 407 -- 411.

\bibitem{LL} L. Lovász. Problem $11$, in Combinatorial structures and their applications. In \emph{Proc. Calgary Internat.
Conf. (Calgary, Alberta, 1969)}, pages xvi+508, New York, 1970. Gordon and Breach Science Publishers.

\bibitem{mogh} M. Moghimian, \emph{An upper bound for the regularity of powers of edge ideals}, Bull. Iranian Math. Soc. Vol. 43 (2017), No. 6, pp. 1695--1698. 

\bibitem{musu} T. M{\"u}tze and P. Su, \emph{Bipartite Kneser graphs are Hamiltonian}, Combinatorica {\bf 37(6)} (2017) 1206--1219.

\bibitem{shpa} S. A. Rather and P. Singh, \emph{Graded Betti numbers of crown edge ideals}, Communications in Algebra,{\bf 47}(4), (2019), 1690-1698. DOI:10.1080/00927872.2018.1513018.

\bibitem{pash} P. Singh and S. A. Rather, \emph{On minimal free resolution of edge ideals of multipartite crown graphs}, Commun. Algebra, 48:3(2020), 1314--1326.

\bibitem{paro} P. Singh and R. Verma (2020), \emph{Betti numbers of edge ideals of some split graphs}, Communications in Algebra, DOI: 10.1080/00927872.2020.1777559.



\bibitem{Vill} R. H. Villarreal, \emph{Monomial Algebras}, second edition, CRC Press, Taylor \& Francis group, 2015 .

\bibitem{wood} R. Woodroofe, \emph{Matchings, coverings, and Castelnuovo-Mumford regularity}, J. Commut. Algebra {\bf 6} (2014), no. 2, 287--304.



 

\end{thebibliography}
\end{document}